\documentclass[10pt]{article}
\usepackage{amsmath,amssymb,amsthm}
\usepackage{epsfig}
\usepackage{color}

\title{Bounds on the nonnegative signed domination number of graphs}

\author {
Doost Ali Mojdeh, Babak Samadi\thanks{Corresponding author}\\
Department of Mathematics\\
University of Mazandaran, Babolsar, Iran\\
{\tt damojdeh@umz.ac.ir}\\
{\tt samadibabak62@gmail.com$^*$}\\
and\\
Lutz Volkmann\\
Lehrstuhl II f\"{u}r Mathematik\\
RWTH Aachen University, 52056 Aachen, Germany\\
{\tt volkm@math2.rwth-aachen.de}\vspace{3mm}\\
}
\date{}

\newtheorem{theorem}{Theorem}[section]

\newtheorem{lemma}[theorem]{Lemma}
\newtheorem{Observation}[theorem]{Observation}

\theoremstyle{definition}

\begin{document}

\maketitle

\begin{abstract}

\noindent \ \ The aim of this work is to investigate the nonnegative signed domination number $\gamma^{NN}_s$ with emphasis on regular, ($r+1$)-clique-free graphs and trees. We give lower and upper bounds on $\gamma^{NN}_s$ for regular graphs and prove that $n/3$ is the best possible upper bound on this parameter for a cubic graph of order $n$, specifically. As an application of the classic theorem of Tur\'{a}n we bound $\gamma^{NN}_s(G)$ from below, for an ($r+1$)-clique-free graph $G$ and characterize all such graphs for which the equality holds, which corrects and generalizes a result for bipartite graphs in [Electron. J. Graph Theory Appl. 4 (2) (2016), 231--237], simultaneously. Also, we bound $\gamma^{NN}_s(T)$ for a tree $T$ from above and below and characterize all trees attaining the bounds.\vspace{.5mm}\vspace{.5mm}
\end{abstract}


\section{Introduction and preliminaries}

\ \ Throughout this paper, let $G$ be a finite simple graph with vertex set $V=V(G)$ and edge set $E=E(G)$. We use \cite{w} as a reference for terminology and notation which are not defined here. The {\em open neighborhood} of a vertex $v$ is denoted by $N(v)$, and the {\em closed neighborhood} of $v$ is $N[v]=N(v)\cup\{v\}$. The {\em minimum} and {\em maximum degree} of $G$ are respectively denoted by $\delta=\delta(G)$ and $\Delta=\Delta(G)$. The {\em corona} of two graphs $G_{1}$ and $G_{2}$ is the graph $G_{1}\circ G_{2}$ formed from one copy of $G_{1}$ and $|V(G_{1})|$ copies of $G_{2}$ where the $ith$ vertex of $G_{1}$ is adjacent to every vertex in the $ith$ copy of $G_{2}$.

Let $S\subseteq V(G)$. For a real-valued function $f:V(G)\rightarrow \mathbb{R}$ we define $f(S)=\sum_{v\in S}f(v)$. Also, $\omega(f)=f(V(G))$ is the weight of $f$. A {\em signed dominating function} ({\em signed 2-independence function}), abbreviated SDF (S2IF), of $G$ is a function $f:V(G)\rightarrow\{-1,1\}$ such that $f(N[v])\geq1$ ($f(N[v])\leq1$), for every $v\in V(G)$. The {\em signed domination number} ({\em signed 2-independence number}), abbreviated  SDN (S2IN), of $G$ is $\gamma_{s}(G)=\min\{f(V(G))|f \mbox{ is a SDF of}\ G\}$ ($\alpha_{s}^{2}(G)=\max\{f(V(G))|f \mbox{ is a S2IF of}\ G\}$). These parameters were introduced in \cite{dthhs} and \cite{z}, respectively.

A set $S\subseteq V(G)$ is a {\em dominating set} if each vertex in $V(G)\backslash S$ has at least one neighbor in $S$. 
Gallant et al. \cite{gghr} introduced the concept of {\em limited packing} in graphs. They exhibited some real-world applications of it to network security, NIMBY, market saturation and codes. In fact as it is defined in \cite{gghr}, a set of vertices $B\subseteq V(G)$ is called a {\em $k$-limited packing} in $G$ provided that for all $v\in V(G)$, we have $|N[v]\cap B|\leq k$. The {\em limited packing number}, denoted $L_{k}(G)$, is the largest number of vertices in a $k$-limited packing set. In \cite{hh}, Harary and Haynes introduced the concept of {\em tuple domination} in graphs. A set $D\subseteq V(G)$ is a {\em $k$-tuple dominating set} in $G$ if $|N[v]\cap D|\geq k$, for all $v\in V(G)$. The {\em $k$-tuple domination number}, denoted $\gamma_{\times k}(G)$, is the smallest number of vertices in a $k$-tuple dominating set. In fact the authors showed that every graph $G$ with $\delta\geq k-1$ has a $k$-tuple dominating set and hence a $k$-tuple domination number.

A function $f:V(G)\rightarrow\{-1,1\}$ is said to be a {\em nonnegative signed dominating function} (NNSDF) of $G$ if $f(N[v])\geq0$ for each $v\in V(G)$. The {\em nonnegative signed domination number} (NNSDN) of $G$, $\gamma^{NN}_{s}(G)$, is the minimum weight of an NNSDF of $G$. This concept was introduced in \cite{hfx}. For more information the reader can consult \cite{as}.

In this paper, we continue the investigating of the concept of nonnegative signed domination in graphs. In section $2$, we present sharp lower and upper bounds on NNSDN of regular graphs, by using the properties of the above graph parameters. Specifically, we prove that $\gamma^{NN}_{s}(G)\leq n/3$ for a cubic graph $G$ of order $n$. In section $3$, we show that the lower bound $2(-1+\sqrt{1+2n})-n$ for NNSDN of a bipartite graph $G$ of order $n$, given in \cite{as}, is not true as it stands. We correct it by giving a more general result on ($r+1$)-clique-free graphs ($r\geq2$) as an application of the well-known theorem of Tur\'{a}n from the extremal graph theory. Also, we characterize all such graphs attaining the new bound. Finally, in section 4 we give lower and upper bounds on NNSDN with emphasis on trees as:
\begin{center}$-n+2\lceil\frac{\Delta+1}{2}\rceil\leq \gamma^{NN}_s(T)$,\ and $\gamma^{NN}_s(T)\leq n-\ell-s'$ for $n\geq3$\end{center}
where $\ell$ and $s'$ are the number of leaves and the support vertices with odd number of leaves, respectively. Moreover, we give the characterizations of all trees attaining these bounds.

For convenience, throughout the paper we make use of the following notation. Let $f:V(G)\rightarrow\{-1,1\}$ be an NNSDF of graph $G$. Define $V_{+}$ and $V_{-}$ as the set of all vertices of $G$ that are assigned $1$ and $-1$ under $f$, respectively. We consider $[V_{-},V_{+}]$ as the set of edges having one end point in $V_{-}$ and the other in $V_{+}$.


\section{Regular graphs}

\ \ Favaron \cite{f} and Wang \cite{wa} proved that for any $r$-regular graph $G$,
\begin{equation}\label{EQ1}
\gamma_{s}(G)\leq\left \{
\begin{array}{lll}
(\frac{r+1}{r+3})n & \mbox{if} & r\equiv0\ {mod\ 2} \vspace{1.5mm}\\
(\frac{(r+1)^{2}}{r^{2}+4r-1})n & \mbox{if} & r\equiv1\ (mod\ 2)
\end{array}
\right.(\cite {f})
\end{equation}
and
\begin{equation}\label{EQ2}
\alpha_{s}^{2}(G)\geq\left \{
\begin{array}{lll}
(\frac{-r^{2}+r+2}{r^{2}+r+2})n & \mbox{if} & r\equiv0\ (mod\ 2) \vspace{1.5mm}\\
(\frac{1-r}{1+r})n & \mbox{if} & r\equiv1\ (mod\ 2).
\end{array}
\right.(\cite{wa})
\end{equation}
Moreover, they showed that these bounds are sharp. Also, the following sharp lower and upper bounds on $\gamma_{s}(G)$ and $\alpha_{s}^{2}(G)$ of an $r$-regular graph $G$ were given in \cite{dthhs,hw}, and \cite{z}, respectively.
\begin{equation}\label{EQ3}
\gamma_{s}(G)\geq\left \{
\begin{array}{lll}
\frac{n}{r+1} & \mbox{if} & r\equiv0\ (mod\ 2) \vspace{1.5mm}\\
\frac{2n}{r+1} & \mbox{if} & r\equiv1\ (mod\ 2)
\end{array}
\right.(\cite {dthhs,hw})
\end{equation}
and
\begin{equation}\label{EQ4}
\alpha_{s}^{2}(G)\leq\left \{
\begin{array}{lll}
\frac{n}{r+1} & \mbox{if} & r\equiv0\ (mod\ 2) \vspace{1.5mm}\\
0 & \mbox{if} & r\equiv1\ (mod\ 2).
\end{array}
\right.(\cite {z})
\end{equation}\label{EQ5}
We are now in a position to exhibit the main theorem of this section.
\begin{theorem}\label{T2.1}
Let $G$ be an $r$-regular graph of order $n$. Then\vspace{1mm}\\
\emph{(i)}\ $\gamma^{NN}_{s}(G)=n-2L_{\lfloor\frac{r+1}{2}\rfloor}(G)$.\vspace{1mm}\\
\emph{(ii)}\ $\gamma^{NN}_{s}(G)=2\gamma_{\times\lceil\frac{r+1}{2}\rceil}(G)-n$.
\end{theorem}
\begin{proof}
(i) Let $B$ be a maximum $\lfloor\frac{r+1}{2}\rfloor$-limited packing in $G$. We define $f:V(G)\rightarrow\{-1,1\}$ by
$$f(v)=\left \{
\begin{array}{lll}
-1 & \mbox{if} & v\in B \\
\ 1 & \mbox{if} & v\in V(G)\setminus B.
\end{array}
\right.$$
Then $f(N[v])=|N[v]\cap(V(G)\setminus B)|-|N[v]\cap B|=|N[v]|-2|N[v]\cap B|\geq r+1-2\lfloor\frac{r+1}{2}\rfloor\geq0$, for all $v\in V(G)$. Therefore, $f$ is an NNSDF of $G$. So,
\begin{equation*}
\gamma^{NN}_{s}(G)\leq f(V(G))=n-2L_{\lfloor\frac{r+1}{2}\rfloor}(G).
\end{equation*}

If $f$ is a minimum NNSDF, then $|N[v]\cap V_{-}|\leq \lfloor\frac{r+1}{2}\rfloor$, for all $v\in V(G)$. Thus, $V_{-}$ is a $\lfloor\frac{r+1}{2}\rfloor$-limited packing in $G$. Therefore, 
\begin{equation*}
(n-\gamma^{NN}_{s}(G))/2=|V_{-}|\leq L_{\lfloor\frac{r+1}{2}\rfloor}(G).
\end{equation*}
So, the first equality holds.\\
(ii) Suppose that $D$ is a minimum $\lceil\frac{r+1}{2}\rceil$-tuple dominating set in $G$. We define $f:V(G)\rightarrow\{-1,1\}$ by
$$f(v)=\left \{
\begin{array}{lll}
\ 1 & \mbox{if} & v\in D \\
-1 & \mbox{if} & v\in V(G)\setminus D.
\end{array}
\right.$$
Then $f(N[v])=|N[v]\cap D|-|N[v]\cap(V(G)\setminus D)|=2|N[v]\cap D|-|N[v]|\geq 2\lceil\frac{r+1}{2}\rceil-r-1\geq0$, for each vertex $v$. So, $f$ is an NNSDF of $G$. This shows that 
\begin{equation*}
\gamma^{NN}_{s}(G)\leq f(V(G))=2\gamma_{\times\lceil\frac{r+1}{2}\rceil}(G)-n.
\end{equation*}

Let $f$ be a minimum NNSDF of $G$. Then $|N[v]\cap V_{+}|\geq \lceil\frac{r+1}{2}\rceil$, for each vertex $v$. Hence $V_{+}$ is a $\lceil\frac{r+1}{2}\rceil$-tuple dominating set in $G$. It follows that
\begin{equation*}
(n+\gamma^{NN}_{s}(G))/2=|V_{+}|\geq \gamma_{\times\lceil\frac{r+1}{2}\rceil}(G).
\end{equation*}
This completes the proof of (ii).
\end{proof}
If $f$ is an NNSDF of $G$, then $f(N[v])\geq1$ for each vertex $v$ of even degree. Hence,
\begin{equation}
\gamma_s(G)=\gamma^{NN}_s(G)
\end{equation}
when $G$ is an $r$-regular graph and $r$ is even. Similar to Part (ii) of Theorem \ref{T2.1}, an analogous equality for $\alpha_{s}^{2}(G)$ can be proved as follows:
$$\alpha_{s}^{2}(G)=n-2\gamma_{\times\lceil\frac{r}{2}\rceil}(G).$$ 
Therefore,
\begin{equation}\label{EQ6}
\gamma^{NN}_s(G)=-\alpha_{s}^{2}(G)
\end{equation}
when $G$ is an $r$-regular graph and $r$ is odd.\\
By Theorem \ref{T2.1} and the inequalities (1)--(\ref{EQ6}), we conclude the following theorem.
\begin{theorem}\label{T2.2}
For any $r$-regular graph $G$ of order $n$,
$$\left(\frac{1}{r+1}\right)n\leq \gamma^{NN}_{s}(G)
\leq\left(\frac{r+1}{r+3}\right)n,\ \ \ r\equiv0\ (mod\ 2)$$
and
$$0\leq \gamma^{NN}_{s}(G)\leq\left(\frac{r-1}{r+1}\right)n,\ \ \ r\equiv1\ (mod\ 2).$$
Furthermore, these bounds are sharp.
\end{theorem}

Balister et al. \cite{bbg} proved that if $G$ is a cubic graph of order $n$, then $L_{2}(G)\geq n/3$.  Taking into account this fact and using the first part of Theorem \ref{T2.1}, the upper bound $n/2$ given in Theorem \ref{T2.2} for a cubic graph $G$ of order $n$ can be improved as follows.
\begin{theorem}\label{T2.3}
If $G$ is a cubic graph of order $n$, then $\gamma^{NN}_{s}(G)\leq n/3$.
\end{theorem}
The upper bound in Theorem \ref{T2.3} is the best possible. To see this fact, let $G_{6}$ be the graph depicted in the following figure. It is easy to see that $\gamma^{NN}_{s}(G_{6})=2$. By taking multiple copies of $G_{6}$, we have infinite collection of cubic graphs $G$ with $\gamma^{NN}_{s}(G)=|V(G)|/3$.

\begin{picture}(269.518,188.518)(0,0)

\put(131,175){\circle*{6}}
\put(160,175){\circle*{6}}
\put(116,157){\circle*{6}}
\put(175,157){\circle*{6}}
\put(131,139){\circle*{6}}
\put(160,139){\circle*{6}}

\multiput(133,176)(.037,0){670}{\line(2,0){.9}}
\multiput(131,139)(.037,0){670}{\line(2,0){.9}}
\multiput(158,177)(.027,-.029){670}{\line(2,0){.9}}
\multiput(115,157)(.024,.031){670}{\line(2,0){.9}}
\multiput(175,157)(-.024,-.029){670}{\line(2,0){.9}}
\multiput(131,139)(-.024,.028){670}{\line(2,0){.9}}

\multiput(115,157)(.084,0){670}{\line(2,0){.9}}
\multiput(131,139)(.043,.056){670}{\line(2,0){.9}}
\multiput(160,139)(-.044,.053){670}{\line(2,0){.9}}

\end{picture}\vspace{-52mm}\\
\begin{center}
The graph $G_{6}$.
\end{center}


\section{($r+1$)-clique-free graphs}

\ \ We need the following well-known theorem of Tur\'{a}n from the extremal graph theory.
\begin{lemma}\label{L1}
\emph{(\cite{T}, Tur\'{a}n's Theorem)}. If $G$ is an \emph{(}$r+1$\emph{)}-clique-free graph of order $n$, then
$$|E(G)|\le \frac{r-1}{2r}\cdot n^2,$$
with equality if and only if $G$ is the Tur\'{a}n graph $T_{n,r}$ and $r$ divides $n$.
\end{lemma}

The following lower bound was exhibited in \cite{as} for the NNSDN of a bipartite graph $G$ of order $n$.
\begin{equation}\label{EQ9}
\gamma^{NN}_{s}(G)\geq2(-1+\sqrt{1+2n})-n.
\end{equation}
The above inequality is not true as it stands. It is easy to see that the family
$$\Sigma=\{K_{p,p}\circ \overline{K_{p+1}}\ |\ p\geq1\}$$
serves as an infinite family of counterexamples to (\ref{EQ9}) (see Figure $1$).

 In what follows we bound $\gamma^{NN}_s(G)$ from below for all ($r+1$)-clique-free graphs ($r\geq2$). Moreover, as the special case $r=2$ of such graphs, the given lower bound on $\gamma^{NN}_s(G)$ in the next theorem can be considered for a bipartite graph $G$ instead of (\ref{EQ9}).
\begin{theorem}
Let $r\geq2$ be an integer. If $G$ is an \emph{(}$r+1$\emph{)}-clique-free graph of order $n$, then
$$\gamma^{NN}_{s}(G)\geq-\frac{2r}{r-1}+\frac{2}{r-1}\sqrt{r^{2}+r(r-1)n}-n.$$
Furthermore, the equality holds if and only if $G=H\circ \overline{K_{(r-1)p+1}}$ in which $H$ is a complete $r$-partite graph with $p$ vertices in each partite set.
\end{theorem}

\begin{proof}
Let $f$ be a minimum NNSDF of $G$. Then, each vertex $v$ in $V_{-}$ has at least one neighbor in $V_{+}$. Also, $|N(v)\cap V_{-}|\leq|N(v)\cap V_{+}|+1$ for all $v\in V_{+}$. Moreover, by lemma \ref{L1} we have \begin{equation}\label{EQ1}
\begin{array}{lcl}
|V_{-}|&\leq& |[V_{-},V_{+}]|=\sum_{v\in V_{+}}|N(v)\cap V_{-}|\leq \sum_{v\in V_{+}}(|N(v)\cap V_{+}|+1)\\
&=&2|E(G[V_{+}])|+|V_{+}|\leq(r-1)|V_{+}|^2/r+|V_{+}|.
\end{array}
\end{equation}
This follows that
$$(r-1)|V_{+}|^2/r+2|V_{+}|-n\geq0.$$
Solving the above inequality for $|V_{+}|$ we obtain
$$(n+\gamma^{NN}_{s}(G))/2=|V_{+}|\geq\frac{r}{r-1}(-1+\sqrt{1+(r-1)n/r}).$$
This implies the desired lower bound.

Let $G=H\circ \overline{K_{(r-1)p+1}}$. Then, the function $g$ assigning $1$ to the vertices in $V(H)$ and $-1$ to the other vertices defines an NNSDF with weight $-r^{2}p^{2}+rp^{2}$. Therefore, $\gamma^{NN}_s(G)\leq-r^{2}p^{2}+rp^{2}$. On the other hand, since $n=rp+rp((r-1)p+1))$, we have
$$\gamma^{NN}_{s}(G)\geq-\frac{2r}{r-1}+\frac{2}{r-1}\sqrt{r^{2}+r(r-1)n}-n=-r^{2}p^{2}+rp^{2}.$$
Therefore, $\gamma^{NN}_{s}(G)=-r^{2}p^{2}+rp^{2}$.

Conversely, suppose that the equality holds. Then $|V_{-}|=|[V_{-},V_{+}]|$,
$|N(v)\cap V_{-}|=|N(v)\cap V_{+}|+1$ for all
$v\in V_{+}$, and $|E(G[V_{+}])|=(r-1)|V_{+}|^2/2r$, by (\ref{EQ1}).
This implies that every vertex in $V_{-}$ has exactly one neighbor in $V_{+}$ (and therefore the vertices in
$V_{-}$ are independent) and every vertex $v$ in $V_{+}$ has exactly
$|N(v)\cap V_{+}|+1$ neighbors in $V_{-}$. Moreover, $G[V_{+}]$ is the Tur\'{a}n
graph $T_{|V_{+}|,r}$ in which $r\mid|V_{+}|$. This implies that each partite set of it has the cardinality $|V_{+}|/r$.
Thus, $G=G[V_{+}]\circ\overline{K_{(r-1)|V_{+}|/r+1}}$. This completes the proof.
\end{proof}\vspace{23mm}

\begin{picture}(269.518,188.518)(0,0)
\put(102,179){\circle*{6}}
\put(147,179){\circle*{6}}
\put(102,211){\circle*{6}}
\put(147,211){\circle*{6}}
\put(192,179){\circle*{6}}
\put(192,211){\circle*{6}}

\put(216,225){\circle*{6}}
\put(208,232){\circle*{6}}
\put(200,238){\circle*{6}}
\put(192,244){\circle*{6}}

\put(165,244){\circle*{6}}
\put(154,244){\circle*{6}}
\put(142,244){\circle*{6}}
\put(131,244){\circle*{6}}

\put(165,147){\circle*{6}}
\put(154,147){\circle*{6}}
\put(142,147){\circle*{6}}
\put(131,147){\circle*{6}}

\put(216,159){\circle*{6}}
\put(208,154){\circle*{6}}
\put(200,150){\circle*{6}}
\put(192,146){\circle*{6}}

\put(78,224){\circle*{6}}
\put(84,231){\circle*{6}}
\put(91,237){\circle*{6}}
\put(100,243){\circle*{6}}

\put(99,148){\circle*{6}}
\put(91,152){\circle*{6}}
\put(84,156){\circle*{6}}
\put(75,160){\circle*{6}}

\multiput(147,211)(.024,.048){670}{\line(2,0){.9}}
\multiput(147,211)(-.024,.047){670}{\line(2,0){.9}}
\multiput(147,211)(.009,.048){670}{\line(2,0){.9}}
\multiput(147,211)(-.008,.048){670}{\line(2,0){.9}}

\multiput(102,179)(-.004,-.049){670}{\line(2,0){.9}}
\multiput(102,179)(-.04,-.026){670}{\line(2,0){.9}}
\multiput(102,179)(-.02,-.043){670}{\line(2,0){.9}}
\multiput(102,179)(-.027,-.03){670}{\line(2,0){.9}}

\multiput(147,179)(-.025,-.049){670}{\line(2,0){.9}}
\multiput(147,179)(.025,-.049){670}{\line(2,0){.9}}
\multiput(147,179)(-.009,-.049){670}{\line(2,0){.9}}
\multiput(147,179)(.01,-.049){670}{\line(2,0){.9}}

\multiput(192,211)(.037,.022){670}{\line(2,0){.9}}
\multiput(192,211)(-.002,.052){670}{\line(2,0){.9}}
\multiput(192,211)(.025,.033){670}{\line(2,0){.9}}
\multiput(192,211)(.011,.04){670}{\line(2,0){.9}}

\multiput(102,211)(-.005,.052){670}{\line(2,0){.9}}
\multiput(102,211)(-.04,.021){670}{\line(2,0){.9}}
\multiput(102,211)(-.019,.04){670}{\line(2,0){.9}}
\multiput(102,211)(-.029,.03){670}{\line(2,0){.9}}

\multiput(191,179)(-.002,-.052){670}{\line(2,0){.9}}
\multiput(191,179)(.04,-.032){670}{\line(2,0){.9}}
\multiput(191,179)(.012,-.042){670}{\line(2,0){.9}}
\multiput(191,179)(.026,-.04){670}{\line(2,0){.9}}

\multiput(191,179)(0,.052){670}{\line(2,0){.9}}
\multiput(191,179)(-.069,.048){670}{\line(2,0){.9}}
\multiput(191,179)(-.14,.05){670}{\line(2,0){.9}}

\multiput(102,179)(0,.052){670}{\line(2,0){.9}}
\multiput(102,179)(.07,.052){670}{\line(2,0){.9}}
\multiput(102,179)(.142,.052){670}{\line(2,0){.9}}

\multiput(147,179)(-.076,.052){670}{\line(2,0){.9}}
\multiput(147,179)(-.001,.052){670}{\line(2,0){.9}}
\multiput(147,179)(.074,.052){670}{\line(2,0){.9}}

\put(105,171){$1$}
\put(152,174){$1$}
\put(105,211){$1$}
\put(151,209){$1$}
\put(184,171){$1$}
\put(185,213){$1$}

\put(217,227){$-1$}
\put(209,234){$-1$}
\put(201,240){$-1$}
\put(193,246){$-1$}

\put(154,248){$-1$}
\put(143,248){$-1$}
\put(131,248){$-1$}
\put(120,248){$-1$}

\put(156,137){$-1$}
\put(144,137){$-1$}
\put(132,137){$-1$}
\put(121,137){$-1$}

\put(218,151){$-1$}
\put(206,145){$-1$}
\put(197,140){$-1$}
\put(185,137){$-1$}

\put(62,225){$-1$}
\put(68,232){$-1$}
\put(77,240){$-1$}
\put(86,246){$-1$}

\put(89,137){$-1$}
\put(81,142){$-1$}
\put(72,146){$-1$}
\put(62,150){$-1$}

\end{picture}\vspace{-52mm}\\
\begin{center}
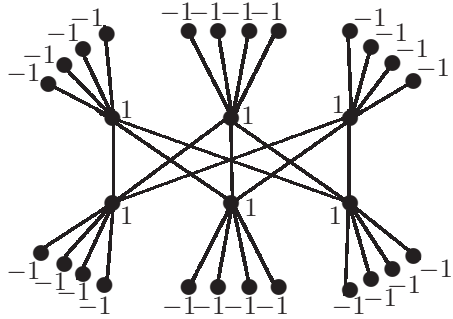
Figure $1$. A member of $\Sigma$ for $p=3$ with $\gamma^{NN}_{s}(G)=-18$.
\end{center}


\section{Trees}

\ \ It has been proved by Henning in \cite{h} that $\alpha_{s}^{2}(T)\geq0$, for any tree $T$. Hence, $\alpha_{s}^{2}(T)$ is bounded from below not depending on the order or any other parameters, for any tree $T$. A similar result cannot be presented for $\gamma^{NN}_{s}(T)$. In fact, the following observation shows that $\gamma^{NN}_{s}(T)$ is not bounded from both above and below.
\begin{Observation}
For any integer $k$, there is a tree $T$ with $\gamma^{NN}_{s}(T)=k$.
\end{Observation}
\begin{proof}
Let $k<0$. Consider $T=P_{|k|}\circ \overline{K_{2}}$. Assigning $-1$ to the leaves and $1$ to the support vertices gives a minimum NNSDF of $T$ with weight $\gamma^{NN}_{s}(T)=k$. On the other hand, it is easy to see that $\gamma^{NN}_{s}(K_{1,2t-1})=0$ for each positive integer $t$. Moreover, $\gamma^{NN}_{s}(P_{n})=n-2\lceil n/3\rceil$ (see \cite{hfx}) shows that $\gamma^{NN}_{s}(P_{3k})=k$, for each positive integer $k$.
\end{proof}
Let $\Theta$ be the collection of all trees $T$ with maximum degree $\Delta(T)$ formed from the star $K_{1,\Delta(T)}$,
with a central vertex $u$ of maximum degree, for which all vertices in $V(T)\setminus N[u]$ are leaves with their support
vertices in a set $S\subseteq N(u)$ with $|S|=\lfloor\frac{\Delta}{2}\rfloor$. Moreover, deg$(s)\leq3$ for all $s\in S$.

\begin{theorem}\label{T1}
For any graph $G$ of order $n$ with maximum degree $\Delta$, $\gamma^{NN}_{s}(G)\geq-n+2\lceil\frac{\Delta+1}{2}\rceil$. In particular, the equality holds for a tree $T$ if and only if $T\in \Theta$.
\end{theorem}
\begin{proof}
Let $f$ be a minimum NNSDF of $G$ and $u$ be a vertex of the maximum degree in $V(G)$. Since $f(N[u])\geq0$, at least $\lceil\frac{\Delta+1}{2}\rceil$ vertices in $N[u]$ belong to $V_{+}$. This implies that\\
$$\gamma^{NN}_{s}(G)=f(V(G))=-n+2|V_{+}|\geq-n+2\left\lceil\frac{\Delta+1}{2}\right\rceil.$$

Suppose that the equality holds for a tree $T$. This shows that $|V_{+}|=\lceil\frac{\Delta+1}{2}\rceil$ and therefore
$V_{+}\subseteq N[u]$. If $\Delta=1$, then $T=K_{2}$. Moreover, it is easy to see that $G$ is formed from $K_{1,2}$ by
adding at most one pendant edge to just one leaf of $K_{1,2}$ if $\Delta=2$. In each of the two cases $T\in \Theta$. So, we may assume that $\Delta\geq3$. Since $\lceil\frac{\Delta+1}{2}\rceil<\Delta$, there exists a vertex $w$ in $N(u)$ with $f(w)=-1$. Therefore, $f(u)=1$ and the other $\lceil\frac{\Delta+1}{2}\rceil-1=\lfloor\frac{\Delta}{2}\rfloor$ vertices of $V_{+}$ appears in $N(u)$. Since $T$ is a tree, the subset $N(u)$ is independent and every vertex in $V(G)\setminus N[u]$ has at most one neighbor in $N(u)$. On the other hand, the condition $f(N[v])\geq0$ for each vertex $v$, implies that every vertex not in $N[u]$ has exactly one neighbor in $S=N(u)\cap V_{+}$, with $|S|=\lfloor\frac{\Delta}{2}\rfloor$. Furthermore, every vertex in $S$ has at most two neighbors except $u$ and all vertices in $V(G)\setminus N[u]$ are leaves with their support vertices in $S$. Hence, $T\in \Theta$.

Now let $T\in \Theta$ and $u$ be a vertex of maximum degree. It is easy to see that the function $f$ assigning 1 to all vertices of $S\cup\{u\}$ and $-1$ to the other vertices defines an NNSDF of $T$ with weight $-n+2\lceil\frac{\Delta+1}{2}\rceil$. So, $\gamma^{NN}_{s}(T)\leq-n+2\lceil\frac{\Delta+1}{2}\rceil$. This completes the proof.
\end{proof}

Note that the lower bound given in Theorem \ref{T1} is sharp not only for all trees in $\Theta$ but for some other collections of graphs, for example the complete graphs.

For a tree $T$, let $L(T)$ and $S(T)$ be the set of leaves and support vertives, respectively. For any support vertex $v$ of $T$ consider $L_{v}$ as the set of all leaves adjacent to $v$ and $\ell_{v}=|L_{v}|$.\\
For characterizing all trees attaining the next bound we introduce $\Omega$ to be the collection of all trees $T$ satisfying:\vspace{.35mm}\\
(i) $T$ is a star with even number of leaves,\vspace{.4mm}\\
or\vspace{.4mm}\\
(ii) $T$ has no support vertex with even number of leaves and $S(T)$ is a dominating set in $T$ in which everey vertex has
at most one neighbor in $V(T)\setminus L(T)$.
\begin{theorem}
Let $T$ be a tree of order $n\geq3$ with $\ell$ leaves and $s'$ be the number of support vertices with odd number of leaves. Then, $\gamma^{NN}_{s}(T)\leq n-\ell+s'$. Moreover, the equality holds if and only if $T\in \Omega$.	
\end{theorem}
\begin{proof}
Let $f$ be a minimum NNSDF of $T$ and $u$ be a support vertex. Then there exists a vertex $u'$ in $N[u]$ with $f(u')=-1$, for otherwise the function $f'$ assigning $-1$ to a leaf of the support vertex $u$ and $f'(x)=f(x)$ to the other vertices $x$ is an NNSDF with the weight $f'(V(T))<f(V(T))$, a contradiction. Also, without loss of generality, we may assume that $L_{u}\cap V_{-}$ is not empty. Otherwise the function $p$ assigning $-1$ to a leaf of $u$, $1$ to $u'$ and $p(x)=f(x)$ to the other vertices $x$ would be an NNSDF with $p(V(T))=\gamma^{NN}_{s}(T)$ (if $f(u)=-1$, then we consider $u'$ as $u$). Thus, we may always assume that $L_{u}\cap V_{-}\neq \emptyset$ and $f(u)=1$, for all support vertices $u$ of $T$.

Suppose that there exists a support vertex $v$ which is adjacent to at most $\lceil\frac{\ell_{v}}{2}\rceil-1$ vertices in $V_{-}$. Then, $f(N[v])\geq|N[v]|-2(\lceil \ell_{v}/2\rceil-1)\geq2$. Let $v'$ be a leaf adjacent to $v$ with $f(v')=1$. Then, it is easy to see that the function defined by $g(v')=-1$ and $g(x)=f(x)$ for each $x\in V(T)\setminus\{v'\}$ is an NNSDF of $T$ with weight $g(V(T))<f(V(T))$, a contradiction. Therefore, every support vertex $v$ is adjacent to at least $\lceil\frac{\ell_{v}}{2}\rceil$ neighbors in $V_{-}$. Moreover, without loss of generality, we may assume that these $\lceil\frac{\ell_{v}}{2}\rceil$ neighbors belong to $L_{v}$.

Let $S(T)=\{v_{1},...,v_{s}\}$ and $S'=\{v_{1},...,v_{s'}\}$ be the set of support vertices with odd number of leaves.
Then\\
$\gamma^{NN}_{s}(T)=n-2|V_{-}|\leq n-2\left(\lceil\frac{\ell_{v_{1}}}{2}\rceil+...
+\lceil\frac{\ell_{v_{s'}}}{2}\rceil+\lceil\frac{\ell_{v_{s'+1}}}{2}\rceil+...+\lceil\frac{\ell_{v_{s}}}{2}\rceil\right)$

\begin{center}$=n-2\left(\frac{\ell_{v_{1}}+1}{2}+...+\frac{\ell_{v_{s'}}+1}{2}+\frac{\ell_{v_{s'+1}}}{2}+...
+\frac{\ell_{v_{s}}}{2}\right)=n-(\ell+s').$\end{center}

Let $\gamma^{NN}_s(T)=n-\ell-s'$ and $f$ be a minimum NNSDF of $T$ which assigns $-1$ to exactly $\lceil\frac{\ell_{v}}{2}\rceil$ leaves for each support vertex $v$ and $1$ to other vertices. Assume that $T$ is not a star with even number of leaves. We prove that $T$ satisfies (ii). If $T$ has a support vertex $v$ with even number of leaves, then the function $g$ assigning $-1$ to exactly $\lceil\frac{\ell_{v}}{2}\rceil+1$ leaves of $v$, $1$ to its other leaves and $g(x)=f(x)$ to the other vertices $x$ would be an NNSDF of $T$ with weight $g(V(T))<f(V(T))$ contradicting the fact that $f(V(T))=\gamma^{NN}_s(T)$. Therefore, $T$ has no support vertex with even number of leaves. We now show that $S(T)$ is a dominating set in $T$. If a vertex $w$ in $V(T)\setminus(S(T)\cup L(T))$ has no neighbor in $S(T)$, then $k(w)=-1$ and $k(x)=f(x)$ for $x\in V(T)\setminus\{w\}$ would be an NNSDF with $k(V(T))< f(V(T))$. This contradiction implies that $S(T)$ is a dominating set. Suppose to the contrary that there exists a vertex $u\in S(T)$ such that $|N(u)\cap(V(T)\setminus L(T))|\geq2$. Then the function $h$ assigning $-1$ to exactly $\lceil\frac{\ell_{u}}{2}\rceil+1$ leaves of $u$, $1$ to its other leaves and $h(x)=f(x)$ for $x\in V(T)\setminus L_{u}$ defines an NNSDF of $T$ with weight $h(V(T))<f(V(T))$, a contradiction. The above argument shows that $T$ satisfies (ii).

Suppose that $T\in \Omega$ and $f$ is a minimum NNSDF of $T$ which assigns $1$ to $v$ and $-1$ to at least $\lceil\frac{\ell_{v}}{2}\rceil$ leaves of each support vertex $v$. Both (i) and (ii) shows that $f$ assigns $-1$ to exactly $\lceil\frac{\ell_{v}}{2}\rceil$ leaves of each support vertex $v$. Now let $f(u)=-1$ for some vertex $u$ in $V(T)\setminus(S(T)\cup L(T))$. Since $S(T)$ dominates $V(T)$, there exists a vertex $v\in S(T)$ adjacent to $u$. Then (ii) implies that $N(v)\cap(V(T)\setminus(S(T)\cup L(T)))=\{u\}$ and hence $f(N[v])\leq-1$, a contradiction. Therefore, $f$ assigns $1$ to all vertices in $V(T)\setminus(S(T)\cup L(T))$. Thus, $\gamma^{NN}_s(T)=n-\ell-s'$. This completes the proof.
\end{proof}


\section{Concluding remarks}

\ \ As applications of the concepts of limited packing and tuple domination we exhibited sharp lower and upper bounds
on NNSDN of regular graphs. We made use of  Tur\'{a}n's Theorem for bounding the NNSDN of $(r+1)$-clique-free
graphs from below, and bounded this parameter for trees and characterized all graphs (trees) attainig the bounds.
It is worth giving exact formulas or bounds for this parameter of some other certain families of graphs. For example,
grids, nearly regular graphs, claw-free graphs, etc. We now conclude the paper with the following problem:\vspace{0.5mm}\\
\textbf{Problem.} How can we classify the other families of graphs by NNSDN?



\begin{thebibliography}{}

\bibitem {as} M. Atapour and S.M. Sheikholeslami, {\em On the nonnegative signed domination numbers in graphs}, Electron. J. Graph Theory Appl. {\bf 4} (2) (2016), 231--237.
\bibitem {bbg} P.N. Balister, B. Bollobas and K. Gunderson, {\em Limited packings of closed neighbourhoods in graphs}, arXiv: 1501.01833v1 [math.CO] 8 Jan 2015.
\bibitem {dthhs} J. E. Dunbar, S. T. Hedetniemi, M. A. Henning and P. J. Slater, {\em Signed domination in graphs}, in: Graph Theory, Combinatorics, and Applications, (John Wiley $\&$ Sons, 1995), 311-322.
\bibitem {f} O. Favaron, {\em Signed domination in regular graphs}, Discrete Math. {\bf 158} (1996), 287--293.
\bibitem {gghr} R. Gallant, G. Gunther, B.L. Hartnell and D.F. Rall, {\em Limited packing in graphs}, Discrete Appl. Math. {\bf 158} (2010), 1357--1364.
\bibitem {hfx} Z. Huang, Z. Feng and H. Xing, {\em On nonnegative signed domination in graphs and its algorithmic complexity}, J. Networks {\bf 8} (2013), 365--372.
\bibitem {hh} F. Harary and T.W. Haynes, {\em Double domination in graphs}, Ars Combin. {\bf 55} (2000), 201--213.
\bibitem {hw} R. Haas and T.B. Wexler, {\em Bounds on the signed domination number of a graph},  Electron. Notes Discrete Math. {\bf 11} (2002), 742-750.
\bibitem {h} M.A. Henning, {\em Signed 2-independence in graphs}, Discrete Math. {\bf 250} (2002), 93--107.
\bibitem {T} P. Tur\'{a}n, {\em On an extremal problem in graph theory}, Math. Fiz. Lapok. {\bf 48} (1941), 436--452.
\bibitem {wa} C. Wang, {\em Voting against in regular and nearly regular graphs} Appl. Anal. Discrete Math. {\bf 4} (2010), 207--218.
\bibitem {w} D.B. West, Introduction to Graph Theory (Second Edition), Prentice Hall, USA, 2001.
\bibitem {z} B. Zelinka, {\em On signed 2-independence number of graphs}, manuscript.
\end{thebibliography}
\end{document}